\documentclass[12pt,reqno]{amsart}
\usepackage{latexsym,amsmath,amsfonts,amssymb,amsthm}
\theoremstyle{plain}
\newtheorem{Thm}{Theorem}

\newtheorem{Cor}{Corollary}
\newtheorem{Conj}{Conjecture}
\theoremstyle{definition}
\newtheorem*{Ack}{Acknowledgment}
\theoremstyle{remark}

\def\Z{\mathbb Z}

\def\N{\mathbb N}

\def\1{{\bf 1}}

\def\pmod #1{\ ({\rm{mod}}\ #1)}

\begin{document}
\title{Proof of a conjecture of Sun}
\author{Hao Pan}
\email{haopan79@yahoo.com.cn}
\address{Department of Mathematics, Nanjing University,
Nanjing 210093, People's Republic of China} 
\maketitle
\begin{abstract}
We confirm a conjecture of Sun.
\end{abstract}

Recently, Z.-W. Sun \cite{Sun} proved that for any $k\geq1$,
$$
\frac{1}{(2^k-2)n+1}\binom{(2^k-1)n}{n}\binom{2(2^k-1)n}{(2^k-1)n}
$$
is divisible by
$$
2^{k-1}\binom{2n}{n}.
$$
One key of Sun's proof is the following lemma:

\medskip{\it For positive integers $n$ and $k$, the number of $1$'s in the binary expansion of $(2^k-1)n$ is at least $k$.}\medskip

\noindent In fact, Sun got a stronger result:

\medskip{\it For a prime $p$ and positive integers $n$ and $k$, The sum of all digits in the expansion of $(p^k-1)n$ in base $p$ is at least $k(p-1)$.}\medskip

\noindent Motivated by the above results, Sun made the following conjecture.
\begin{Conj}\label{c1} 
{\rm (I)} Suppose that $n,m,k$ are positive integers and $m\geq 2$. Then there are at least $k$ non-zero digits in the expansion of $\frac{m^k-1}{m-1}n$ in base $m$.

{\rm (II)} Suppose that $n,m,k$ are positive integers and $m\geq 2$. Then the sum of all digits in the expansion of $(m^k-1)n$ in base $m$ is at least $k(m-1)$.
\end{Conj}
In this short note, we shall confirm  Conjectures \ref{c1}.

\begin{Thm}\label{t1}
Suppose that $m\geq 2$, $k\geq 1$ and $a_1,\ldots,a_k\in\N$ are not all zero. Let $d$ be a divisor of $m^k-1$. Suppose that $\tau(x_1,\ldots,x_n)$ is a nonegative integer-valued symmetric function satisfying that
$$
\tau(mq+b,x_2,\ldots,x_{k-1},x_k)\geq\tau(b,x_2,\ldots,x_{k-1},x_k+q)
$$
for any $q\geq 1$ and $0\leq b<m$. If
$$
a_1m^{k-1}+a_2m^{k-2}+\cdots+a_{k-1}m+a_k\equiv0\pmod{d},
$$
then
$$\tau(a_1,a_2,\ldots,a_n)\geq\min_{1\leq t\leq (m^k-1)/d}\{\tau^\circ(td)\},$$
where $$
\tau^\circ(h)=\tau(c_1,c_2,\ldots,c_k)
$$
if $0\leq h<m^k$ has an $m$-adic expansion $h=c_1m^{k-1}+c_2m^{k-2}+\cdots+c_k$.
\end{Thm}
\begin{proof}
Let
$$
S=\{(a_1,\ldots,a_k):\, a_1,\ldots,a_k\in\N\text{\ are not all zero},\ \sum_{j=1}^ka_jm^{k-j}\equiv0\pmod{d}\}.
$$
Since
$$
m\bigg(\sum_{j=1}^ka_jm^{k-j}\bigg)=\sum_{j=1}^ka_jm^{k-j+1}\equiv a_1+\sum_{j=1}^{k-1}a_{j+1}m^{k-j}\pmod{d},
$$
$(a_1,a_2,\ldots,a_k)\in S$ implies that $(a_2,a_3,\ldots,a_k,a_1)\in S$. For ${\bf x}=(a_1,\ldots,a_k)$, define
$$\sigma({\bf x})=a_1+\cdots+a_k$$
Let
$$
S^*=\{(a_1,\ldots,a_k)\in S:\,\tau(a_1,\ldots,a_k)=\min_{(a_1,\ldots,a_k)\in S}\{\tau(a_1,\ldots,a_k)\}\}
$$
Choose an ${\bf x}=(a_1,\ldots,a_k)\in S^*$ such that $$\sigma({\bf x})=\min_{(a_1,\ldots,a_k)\in S^*}\{\sigma(a_1,\ldots,a_k)\}.$$
And noting that $\tau$ is symmetric, without loss of generality, we may assume that $a_1\geq \max\{a_2,a_3,\ldots,a_k\}$. We shall prove that $a_1<m$. Assume on the contrary that $a_1\geq m$. Write $a_1=mq+b$ with $0\leq b<m$ and $q\geq 1$. Then
$$
a_1m^{k-1}=(mq+b)m^{k-1}\equiv bm^{k-1}+q\pmod{d}.
$$
Hence ${\bf x}^*=(b,a_2,\ldots,a_{k-1},a_k+q)\in S$. Note that now
$$
\tau(a_1,\ldots,a_k)=\tau(mq+b,a_2,\ldots,a_{k-1},a_k)\geq\tau(b,a_2,\ldots,a_{k-1},a_k+q)
$$
Hence ${\bf x}^*$ also lies in $S^*$. But clearly
$$\sigma({\bf x})-\sigma({\bf x}^*)=a_1+a_k-(b+a_k+q)=(m-1)q\geq 1,
$$
i.e., $\sigma({\bf x}^*)<\sigma({\bf x})$. This evidently leads to a contradiction with the choice of ${\bf x}$.

So we must have $a_1<m$, i.e., $\max\{a_1,a_2,\ldots,a_k\}\leq m-1$. Thus 
$a_1m^{k-1}+\cdots+a_k=t_0d$ with $1\leq t_0\leq(m^k-1)/d$, and
$$\tau(a_1,a_2,\ldots,a_n)=\tau^\circ(t_0d)\geq\min_{1\leq t\leq (m^k-1)/d}\{\tau^\circ(td)\}.$$
\end{proof}
\begin{Cor}\label{cr1}
Suppose that $m\geq 2$, $k\geq 1$ and $a_1,\ldots,a_k\in\N$ are not all zero. If
$$
a_1m^{k-1}+a_2m^{k-2}+\cdots+a_{k-1}m+a_k\equiv0\pmod{\frac{m^k-1}{m-1}},
$$
then
$$\sum_{j=1}^k\left\lceil\dfrac{a_j}{m}\right\rceil\geq k,$$
where $\lceil x\rceil=\min\{z\in\Z:\,z\geq x\}$.
\end{Cor}
\begin{proof} Let
$$\tau(x_1,\ldots,x_k)=\sum_{j=1}^k\left\lceil\dfrac{x_j}{m}\right\rceil.$$
Note that
$$
\left\lceil\dfrac{mq+b}{m}\right\rceil=q+\left\lceil\dfrac{q+b}{m}\right\rceil\geq q+\left\lceil\dfrac{b}{m}\right\rceil,
$$
and
$$
\left\lceil\dfrac{x_k+q}{m}\right\rceil\leq\left\lceil\dfrac{x_k}{m}\right\rceil+\left\lceil\dfrac{q}{m}\right\rceil\leq\left\lceil\dfrac{x_k}{m}\right\rceil+q.
$$
We have $\tau(mq+b,x_2,\ldots,x_{k-1},x_k)\geq\tau(b,x_2,\ldots,x_{k-1},x_k+q)$. Hence $\tau$ satisfies the requirements of Theorem \ref{t1}. Thus by Theomre \ref{t1},
\begin{align*}
\sum_{j=1}^k\left\lceil\dfrac{a_j}{m}\right\rceil=&\tau(a_1,\ldots,a_k)\geq\min_{1\leq t\cdot\leq m-1}\bigg\{\tau^\circ\bigg(t\frac{m^k-1}{m-1}\bigg)\bigg\}\\
\geq&\min_{1\leq t\cdot\leq m-1}\bigg\{\tau^\circ\bigg(\sum_{j=1}^k tm^{k-j}\bigg)\bigg\}=\min_{1\leq t\cdot\leq m-1}\bigg\{\sum_{j=1}^k\left\lceil\frac{t}{m}\right\rceil\bigg\}=k.
\end{align*}
\end{proof}

Let us explain why Corollary \ref{cr1} implies Part (I) of Conjecture \ref{c1}. White 
$$n\cdot\frac{m^k-1}{m-1}=b_1m^h+b_2m^{h-1}+\cdots+b_{h-1}m+b_h$$ 
with $0\leq b_i<m$. Let
$$
a_j=\sum_{\substack{1\leq i\leq h\\ h+1-i\equiv k+1-j\pmod{k}}}b_i.
$$
Since all $b_i$ is less than $m$, we have
$$
\left\lceil\dfrac{a_j}{m-1}\right\rceil\leq|\{1\leq i\leq h:\, h+1-i\equiv k+1-j\pmod{k},\ b_i>0\}|.
$$
On the other hand,
$$
\sum_{i=1}^hb_im^{h+1-i}\equiv\sum_{j=1}^ka_jm^{k+1-j}\equiv0\pmod{(m^k-1)/(m-1)}.
$$
It follows from Theorem \ref{t1} that
$$
|\{1\leq i\leq h:\, b_i>0\}|\geq\sum_{j=1}^k\left\lceil\dfrac{a_j}{m-1}\right\rceil\geq\sum_{j=1}^k\left\lceil\dfrac{a_j}{m}\right\rceil\geq k.
$$

Furthermore, Part (II) of Conjecture \ref{c1} is an immediate consequence of the following corollary.
\begin{Cor}\label{cr2}
Suppose that $m\geq 2$, $k\geq 1$ and $a_1,\ldots,a_k\in\N$ are not all zero. If
$$
a_1m^{k-1}+a_2m^{k-2}+\cdots+a_{k-1}m+a_k\equiv0\pmod{m^k-1},
$$
then
$$\sum_{j=1}^k\left\lfloor\dfrac{a_j}{m-1}\right\rfloor\geq k,$$
where $\lfloor x\rfloor=\max\{z\in\Z:\,z\leq x\}$.
\end{Cor}
\begin{proof} Since the case $m=2$ easily follows from Corollary \ref{cr1}, we may assume that $m\geq 3$. Let
$$\tau(x_1,\ldots,x_k)=\sum_{j=1}^k\left\lfloor\dfrac{x_j}{m-1}\right\rfloor.$$
Since
$$
\left\lfloor\dfrac{mq+b}{m-1}\right\rfloor\geq q+\left\lfloor\dfrac{b}{m-1}\right\rfloor,
$$
and
$$
\left\lfloor\dfrac{x_k+q}{m-1}\right\rfloor\leq\left\lfloor\dfrac{x_k}{m-1}\right\rfloor+\left\lfloor\dfrac{q}{m-1}\right\rfloor+1\leq\left\lfloor\dfrac{x_k}{m-1}\right\rfloor+q,
$$
we have $\tau(mq+b,x_2,\ldots,x_{k-1},x_k)\geq\tau(b,x_2,\ldots,x_{k-1},x_k+q)$. Thus $\tau$ satisfies the requirements of Theorem \ref{t1}. And applying Theorem \ref{t1}, we get
$$
\sum_{j=1}^k\left\lfloor\dfrac{a_j}{m-1}\right\rfloor\geq\tau^\circ(m^k-1)=\tau^\circ\bigg(\sum_{j=1}^k (m-1)\cdot m^{k-j}\bigg)\\
=\sum_{j=1}^k\left\lfloor\dfrac{m-1}{m-1}\right\rfloor=k.
$$
\end{proof}
\begin{Ack}I am grateful to Professor Zhi-Wei Sun for his very helpful suggestions on this paper.\end{Ack}

\end{document}